\pdfoutput=1
\RequirePackage{ifpdf}
\ifpdf 
\documentclass[pdftex]{sigma}
\else
\documentclass{sigma}
\fi

\usepackage{comment}

\numberwithin{equation}{section}

\newtheorem{Theorem}{Theorem}[section]

\newtheorem{Lemma}[Theorem]{Lemma}

 { \theoremstyle{definition}
\newtheorem{Definition}[Theorem]{Definition}

\newtheorem{Remark}[Theorem]{Remark} }

\newcommand{\de}{\partial}

\newcommand{\Z}{\mathbb Z}

\newcommand{\C}{\mathbb C}

\newcommand{\al}{\alpha}
\newcommand{\be}{\beta}
\newcommand{\ga}{\gamma}

\newcommand{\la}{\lambda}

\newcommand{\si}{\sigma}

\newcommand{\ord}{\operatorname{ord}}

\begin{document}

\newcommand{\arXivNumber}{2101.01138}

\renewcommand{\PaperNumber}{052}

\FirstPageHeading

\ShortArticleName{Centralizers of Rank One in the First Weyl Algebra}

\ArticleName{Centralizers of Rank One in the First Weyl Algebra}

\Author{Leonid MAKAR-LIMANOV~$^{\rm ab}$}

\AuthorNameForHeading{L.~Makar-Limanov}

\Address{$^{\rm a)}$~Department of Mathematics, Wayne State University, Detroit, MI 48202, USA}
\Address{$^{\rm b)}$~Department of Mathematics \& Computer Science, The Weizmann Institute of Science,\\
\hphantom{$^{\rm b)}$}~Rehovot 76100, Israel}
\EmailD{\href{mailto:lml@wayne.edu}{lml@wayne.edu}}

\ArticleDates{Received January 07, 2021, in final form May 12, 2021; Published online May 19, 2021}

\Abstract{Centralizers of rank one in the first Weyl algebra have genus zero.}


\Classification{16S32}

\section{Introduction}

Take an element $a$ of the first Weyl algebra $A_1$. {\it Rank} of the centralizer $C(a)$ of this element is
the greatest common divisor of the orders of elements in $C(a)$ (orders as differential operators).

This note contains a proof of the following

\begin{Theorem}If the centralizer $C(a)$ of an element $a \in A_1\setminus K$, where $A_1$ is the first Weyl algebra defined over a field~$K$
of characteristic zero, has rank $1$, then $C(a)$ can be embedded into a polynomial ring $K[z]$.
\end{Theorem}

The classical works of Burchnall and Chaundy where the systematic research
of commuting differential operators was initiated are also devoted primarily to
the case of rank $1$ but the coefficients of the operators considered by them
are analytic functions.
Burchnall and Chaundy treated only monic differential operators which
doesn't restrict generality if the coefficients are analytic functions.
Situation is completely different if the coefficients are polynomial.

\section[First Weyl algebra A1 and its skew field of fractions D1]
{First Weyl algebra $\boldsymbol{A_1}$ and its skew field of fractions $\boldsymbol{D_1}$}\label{2q0}

Before we proceed with a proof, here is a short refresher on the first Weyl algebra.

\begin{Definition}
The first Weyl algebra $A_1$ is an algebra over a field $K$
generated by two elements (denoted here by $x$ and $\de$) which satisfy a
relation $\de x - x \de = 1$.
\end{Definition}

When characteristic of $K$ is zero $A_1$ has a
natural representation over the ring of polyno\-mi\-als~$K[x]$ by operators of
multiplication by $x$ and the derivative $\de$ relative to $x$. Hence the
elements of the Weyl algebra can be thought of as differential operators
with polynomial coefficients. They can be written as ordinary polynomials
\[
a= \sum c_{i,j}x^i\de^j,\qquad c_{i,j} \in K
\]
with ordinary addition but a~more complicated
multiplication.

Algebra $A_1$ is rather small, its Gelfand--Kirillov dimension is $2$, hence it is a two-sided Ore ring. Because of that it can be embedded in a skew field $D_1$. A detailed discussion of skew fields related to Weyl algebras and their skew fields of fraction, as well as a definition of Gelfand--Kirillov dimension can be found in a paper~\cite{GK}.

If characteristic of $K$ is zero then the centralizer $C(a)$ of any
element $a \in A_1\setminus K$ is a~com\-mutative subalgebra of $A_1$ of the transcendence degree one. This
theorem which was first proved by Issai Schur in 1904 (see~\cite{S}) and by Shimshon Amitsur by purely algebraic methods (see~\cite{A}) has somewhat
entertaining history which is described in~\cite{ML}.

\begin{Definition}
The rank of a centralizer is the greatest common divisors of the orders of~elements of $C(a)$ considered as differential operators,
i.e., of degrees of elements of $C(a)$ rela\-tive~to~$\de$.
\end{Definition}

\section[Leading forms of elements of A1]{Leading forms of elements of $\boldsymbol{A_1}$}\label{3a}

Given $\rho, \sigma \in \Z$ it is possible to define a weight degree function $w$ on $A_1$ by
\begin{gather*}
w(x) = \rho,\qquad
w(\de) = \sigma, \qquad
w\big(x^i\de^j\big) = \rho i + \sigma j,
\\
w(a) = \max w\big(x^i\de^j\big)\mid c _{i,j} \neq 0\qquad \text{for}\quad a = \sum c_{i,j}x^i\de^j.
\end{gather*}

\begin{Definition}
 The leading form $\bar{a}$ of $a$ is
 \[
 \bar{a} = \sum c_{i,j}x^i\de^j\mid w\big(x^i\de^j\big) = w(a).
 \]
\end{Definition}

One of the nice properties of $A_1$ which was used by Dixmier in his seminal
research of the first Weyl algebra (see~\cite{D}, Lemma 2.7) is the following property of the leading
forms of ele\-ments~of~$A_1$:
\begin{itemize}\itemsep=0pt
\item if $\rho + \sigma > 0$ then $\overline{[a,b]} = \big\{\bar{a},\bar{b}\big\}$ for $a, b \in
A_1$, where $[a, b] = ab - ba$ and $\big\{\bar{a},\bar{b}\big\} = \bar{a}_{\de}\bar{b}_x
- \bar{a}_x\bar{b}_\de$ is the standard Poisson bracket of $\bar{a}$, $\bar{b}$
as commutative polynomials ($\bar{a}_{\de}$ etc.\ are the corresponding partial
derivatives), provided $\{\bar{a},\bar{b}\} \neq 0$;

\item if $\big\{\bar{a},\bar{b}\big\} = 0$ and $w(a) \ne 0$ then $\bar{b}$ is proportional over $K$ to a fractional power of $\bar{a}$.
\end{itemize}

The main ingredient of the considerations below is this property of the leading
forms.

To make considerations clearer the reader may use the Newton polygons of elements of~$A_1$.
The Newton polygon of $a \in A_1$ is the convex hull of those points $(i, j)$ on the plane for which~$c_{i,j} \neq 0$. The Newton polygons of elements
of $A_1$ are less sensible than the Newton polygons of~polynomials in two
variables because they depend on the way one chooses to record elements of~$A_1$
but only those edges which are independent of the choice will be used.

\section{Proof of the theorem}\label{pcp}

\subsection*{Case 1:
$\boldsymbol{\displaystyle a = \de^n + \sum\limits_{i=1}^n a_i(x)\de^{n-i}}$}

If $a \in K[\de]$ then $C(a) = K[\de]$, a ring of polynomials in one variable. Otherwise consider the leading form $\alpha$ of $a$ which contains $\de^n$ and is not a
monomial. This form has a non-zero weight and corresponding $\rho$, $\sigma$ satisfy conditions of the Dixmier's lemma (both $\rho$ and $\si$ are positive).

The leading forms of the elements from $C(a)$ are Poisson commutative with $\al$ since these elements commute with $a$. Therefore
they are proportional over $K$ to the fractional powers of $\alpha$ (as a
commutative polynomial). Because the rank of $C(a)$ is $1$ we should have $\alpha = c\big(\de + c_1x^k\big)^n$.

A mapping $\phi$ of $A_1$ to $A_1$ defined by
\[
x \rightarrow x, \qquad \de \rightarrow \de - c_1x^k
\]
is an automorphism of $A_1$. It~is easy to see that the Newton polygon of $\phi(a)$ belongs to the Newton polygon of $a$ and has a~smaller area.

Hence there exists an automorphism
\[
\psi\colon\quad x \rightarrow x, \qquad\de \rightarrow \de + p(x)
\]
such that $\psi(a) \in K[\de]$.
Therefore $C(a) = K\big[\psi^{-1}(\de)\big]$, a polynomial ring in one variable.

\subsection*{Case 2:
$\boldsymbol{\displaystyle a = x^m\de^n + \sum\limits_{i=1}^n a_i(x)\de^{n-i}, \ m > 0}$}

As above, the leading forms of elements of $C(a)$ are proportional to the
fractional powers of the leading form $\al$ of $a$ (as a commutative polynomial) as long as $\al$ is the
leading form of $a$ relative to weights $\rho$, $\sigma$ of $x$ and $\de$
provided $\rho + \sigma > 0$ and the weight of $\al$ is not zero. Because of
that and since the rank is assumed to be one, $n$ divides $m$.

The Newton polygon ${\mathcal N}(a)$ of $a$ has the vertex $(m,n)$.
If we picture ${\mathcal N}(a)$ on a plane where the $x$ axis is horizontal and the $\de$ axis is vertical then $(m,n)$ belongs to two edges: left an right.
It~is also possible that these edges coincide, or that ${\mathcal N}(a)$ consists just of the vertex $(m,n)$.

If the left edge exists, i.e., is not just the vertex $(m,n)$, then the ray with the vertex $(m,n)$ containing this edge cannot intersect the $y$ axis above the origin.
Indeed, assume that this is the case and the point of intersection is $(0, \mu)$, where $\mu$ is a positive rational number, which must be smaller than $n$.

If we take weights $\rho = \mu - n$, $\si = m$ then $\rho + \si = \mu - n + m > 0$ since $m \geq n$ and we can apply the Dixmier's lemma to the corresponding leading form of $a$.
But then
\[
\bar{a} = \big(x^d\de + cx^s\big)^n, \qquad
\text{where}\quad s = \frac{mn}{\mu - n},
\]
which is impossible since $s < 0$.

Therefore $a$ has a non-trivial leading form of zero weight relative to the weights $\rho = -1$, $\si = d$, where $d = \frac{m}{n}$. This form can be the monomial $x^m\de^n$, or a polynomial in $x^d\de$.

\begin{Lemma}\label{lemma1}
If $a$ has the leading form of weight zero relative to the weights
\[
\rho < 0 < \sigma, \qquad
\rho + \sigma \geq 0
\]
then $C(a)$ is a subring of
a ring of polynomials in one variable. $($Here the rank of $C(a)$ is not essential.$)$
\end{Lemma}

\begin{proof}
Any $b \in C(a)$ has a non-zero leading form $\bar{b}$ of weight zero (relative to $\rho$, $\sigma$).
Indeed, if~$\rho + \sigma > 0$ and $w(b) \ne 0$ then $\big\{\bar{a}, \bar{b}\big\} \ne 0$ by the Dixmier's lemma.

If $\rho + \sigma = 0$ then $\bar{a} \in K[x\de]$ and only elements of $K[x\de]$
commute with it (see the remark below).

Hence the restriction map $b \rightarrow \bar{b}$ is an isomorphism. An algebra generated by all $\bar{b}$ is a~subalgebra of~$K[x^{\sigma}\de^{-\rho}]$ if we assume that $\rho, \sigma \in \Z$ and relatively prime.
\end{proof}

\begin{Remark}
 An equality
 \[
 x^i\de^i = (t-i+1)(t-i+2)\cdots(t-i+i),
 \]
 where $t = x\de$ is the Euler operator is easy to check since
 \[
 xt = xx\de = x(\de x - 1) = tx - x = (t-1)x
 \]
and thus
$xp(t)\de = p(t-1)t$ (see similar computations in~\cite{BC2}).
\end{Remark}

\subsection*{Case 3: $\boldsymbol{\displaystyle a = a_0(x)\de^n + \sum\limits_{i=1}^n a_i\de^{n-i}}$}

In this case $a_0 = \al^n$, $\al \in K[x]$. We may assume that $\al \not\in K$ and that $\al(0) = 0$ (applying an automorphism $x\rightarrow x + c$, $\de \rightarrow \de$ if necessary).
We may also assume that the origin is a~vertex of~${\mathcal N}(a)$ since we can replace $a$ by $a + c$, where $c$ is any element of $K$.
Then ${\mathcal N}(a)$ has the horizontal edge with the right vertex $(m,n)$ and the left vertex $(m',n)$ where
$m'$ is divisible by~$n$. As above, the edge with vertices $(m',n)$ and $(0,0)$ corresponds to the leading form of~$a$ of zero weight
and Lemma~\ref{lemma1} shows that $C(a)$ is isomorphic to a subring of $K[x^{d'}\de]$, where~$d' = \frac{m'}{n}$.

This finishes a proof of the theorem.

Since the proof of the theorem turned out to be too simple and too short we can
complement it by an attempt to describe the rank one centralizers more
precisely. In the first case it is already done, the centralizer is isomorphic
to $K[z]$, where $z = \de + p(x)$ for some $p(x) \in K[x]$.

It~would be interesting to describe $a$ for which $C(a)$ is not isomorphic to a polynomial ring. The second case described above provides us with examples of this phenomenon.

\section{Centralizers in Case 2}\label{lie}

Let us call $(m,n)$ the leading vertex of ${\mathcal N}(a)$ and the edges containing this vertex the leading edges. If the extension of the right leading edge intersects the $x$ axis in the point $(\nu, 0)$,
where $\nu > m - n$ and we take $\rho = n$, $\si = \nu - m$ then $\rho + \si = n + \nu - m > 0$ and by the Dixmier's lemma the leading form $\bar{a}$ which corresponds to this weight is $\big(x^d\de + c x^k\big)^n$, where $k \geq d$. Then, similar to the Case~1 we will make an automorphism
\[
x \to x, \qquad\de \to \de - cx^{k-d},
\]
which will collapse the right leading edge to the leading vertex.

After several steps like that we will obtain $\phi(a)$, where $\phi$ is an automorphism of $A_1$, such that the right leading edge of ${\mathcal N}(\phi(a))$ is parallel to the bisectrix of the first quadrant. For this edge $\rho + \si = 0$ and we cannot apply the Dixmier's lemma to the corresponding leading form.

Since $C(a)$ and $C(\phi(a))$ are isomorphic we will assume that the right leading edge of ${\mathcal N}(a)$ is parallel to the bisectrix.

If the left leading edge of ${\mathcal N}(a)$ is not parallel to the bisectrix we can consider the centralizer of $a$ as a subalgebra of $K\big[\de, x, x^{-1}\big]$ and proceed with automorphisms
\[
x \rightarrow x,\qquad
\de \rightarrow \de - c_1x^{k-d},
\]
where $k - d < -1$ since the Dixmier's lemma will be applicable to the corresponding leading forms.

Hence there exists an automorphism
\[
x \rightarrow x, \qquad
\de \rightarrow
\de + q(x)
\]
of $K[\de, x, x^{-1}]$
such that $\psi(a) = x^{m-n}p(t)$, where $t =
x\de$ (here $q(x)$ is a Laurent polynomial while $p(t)$ is a polynomial).

We see that centralizers of elements with the leading vertex $(m,n)$ are isomorphic to centra\-li\-zers of elements $x^{m-n}p(t)$, $p(t) \in K[t]$, $\deg_t(p(t)) = n$.
If $a = x^{m-n}p(t)$ and $m - n = 0$ then $C(a) = K[t]$. Assume now that $m > n$.

If $b \in C(a)$ then we can present $b$ as the sum of forms, homogeneous relative to the weight $w$ given by
\[
w(x) = 1,\qquad
w(\de) = -1\colon\quad b = \sum_i x^ib_i(t).
\]
Since
\[
[a,b] = \sum_i\big[a, x^ib_i(t)\big] =
\sum _i x^{m-n + i}\big(p(t-i)b_i(t) - b_i(t - m + n)p(t)\big) = 0
\]
all $x^ib_i(t) \in C(a)$. Hence $C(a)$ is a linear span of elements, homogeneous relative to the weight~$w$.

\begin{Lemma}\label{lemma2}
If $b \in C(a)$ is a $w$-homogeneous element then $b^\nu = ca^\mu$ for some relatively prime integers $\mu$ and $\nu$, and $c \in K$.
\end{Lemma}

\begin{proof}
The leading vertex of $b$ is $\la(m,n)$, where $\la = \frac{\mu}{\nu}$, $\mu, \nu \in \Z$, $(\mu,\nu) = 1$ since we can apply the Dixmier's lemma to the leading forms of $a$ and $b$ relative to the weight $w_1(x) = 1$, $w_1(\de) = 1$. Therefore $b^\nu$ and $a^\mu$ have the same leading vertex $\mu(m,n)$ and $\deg_x(b^\nu - ca^\mu) < \deg_x(b^\nu)$ with the appropriate choice of $c \in K$.
If $b_1 = b^\nu - ca^\mu \ne 0$ then $b_1$ is a homogeneous element of $C(a)$ and its leading vertex must be proportional to the leading vertex of $a$.
This is impossible since $w(b_1) = \mu w(a) = \mu(m-n)$: if $\xi(m,n)$ is the leading vertex of $b_1$ then $w(b_1) = \xi(m-n) = \mu(m-n)$ and $\xi = \mu$.
\end{proof}

Since the rank of $C(a)$ is $1$ we can find two elements
\[
b_1 = x^{\be_1}q_1(t), \qquad
b_2 = x^{\be_2}q_2(t)\qquad
\text{such that}\quad
\deg_t(b_2) = \deg_t(b_1) + 1.
\]
Then $b = b_2b_1^{-1}$ belongs to $D_1$ (the skew field of fractions of $A_1$), and commutes with $a$. Using a relation $xt = (t-1)x$ we can write that
\[
b = x^{\be_2}q_2(t)\big(x^{\be_1}q_1(t)\big)^{-1} = x^{\be_2}q_2(t)q_1(t)^{-1}x^{-\be_1} = x^{\be_2-\be_1}r(t),
\]
where $r(t) \in K(t)$.

The leading vertex of $b$ can be defined as the difference of the leading vertices of $b_2$ and $b_1$ and is proportional to the leading vertex of $a$. Since $\deg_t(r) = \deg_t(b_2) - \deg_t(b_1) = 1$ this vertex in coordinates $x, t$ is $(d-1,1)$. (Recall that $d = \frac{m}{n}$.)

If $r$ is a polynomial then $C(a) = K\big[x^{d-1}r\big]$; if $d = 1$ then $C(a) = K[t]$; if $r$ is not a~polynomial and $d > 1$ then some powers of $x^{d-1}r$ are polynomials: say, $a = cb^n$, $c\in K$ because considerations of Lemma~\ref{lemma2} are applicable to $w$ homogeneous elements of~$D_1$ commuting with~$a$.

In the last case
$r(t) \in K(t)$ but $r(t)r(t+d - 1)\cdots r(t + (k-1)(d-1)) \in K[t]$ (observe that $tx = x(t + 1)$). We~can reduce this to $r(t)r(t+1)\cdots r(t + k -1) \in K[t]$ by rescaling $t$ and $r$.

By shifting $t$ if necessary we may assume that one of the roots of $r(t)$ is
$0$ and represent $r$ as a product
$r_0r_1$, where all roots and poles of $r_0$ are
in~$\Z$ and all roots and poles of $r_1$ are not in~$\Z$. It~is clear that
\[
r_0(t)r_0(t+1)\cdots r_0(t + k -1) \in K[t]\qquad \text{and}\qquad
r_1(t)r_1(t+1)\cdots r_1(t + k-1) \in K[t].
\]
Since $\deg(r) = 1$ and $\deg(r_i) \geq 0$ (because $k\deg(r_i)
\geq 0$), degree of one of the $r_i$ is equal to zero and
$r_i(t)r_i(t+1)\cdots r_i(t + k - 1) \in K$ for this $r_i$. But then $r_i(t) = r_i(t + k)$
which is impossible for a
non-constant rational function. Since $r_0(0) = 0$ and $r \neq 0$ we see that
$r_1$ is a constant and all roots and poles of $r$ are in $\Z$.

We can assume now that $0$ is the largest root of $r$ and write
\begin{gather*}
r = ts(t),\qquad
\text{where}\quad
s(t) = \frac{\prod_{i=1}^p (t + \la_i)}{\prod_{i=1}^p (t + \mu_i)} \in
K(t) \setminus K,
\\
\la_i \in \Z, \qquad
\mu_i \in \Z,\qquad
0 \leq \la_1 \leq \la_2 \leq \dots \leq \la_p, \qquad
\mu_1\leq \mu_2 \leq \cdots \leq \mu_p.
\end{gather*}
If $\mu_1 < 0$ then $r(t)r(t+1)\cdots r(t + k -1)$
would have a pole at $t = -\mu_1$. Hence $\mu_1 > 0$ and all poles of $s(t)$
are negative integers while all zeros of $s(t)$ are non-positive integers.

A fraction $\frac{t + \la_i}{t + \mu_i}$ can be presented as $\frac{f_i(t)}{f_i(t+1)}$ if $\la_i < \mu_i$ or as $\frac{f_i(t+1)}{f_i(t)}$ if $\la_i > \mu_i$: indeed,
\[
\frac{t + d}{t} = \frac{(t+1) (t+2) \cdots(t+d)}{t(t+1)\cdots (t+d-1)}\qquad \text{if}\quad d > 0,
\]
take the reciprocal fraction if $d < 0$. Because of that $s(t)$ can be written as
\[
\frac{s_1(t)s_2(t+1)}{s_1(t+1)s_2(t)}, \qquad
s_i(t) \in K[t].
\]
Write
$s_1(t) = s_3(t)s_4(t)$, $s_2(t) =s_4(t)s_5(t)$,
where $s_4(t)$ is the greatest common divisor of $s_1(t)$ and~$s_2(t)$. Then
\[
s(t) = \frac{s_3(t)s_4(t)s_4(t+1)s_5(t+1)}{s_3(t+1)s_4(t+1)s_4(t)s_5(t)} = \frac{s_3(t)s_5(t+1)}{s_3(t+1)s_5(t)}.
\]

All roots of $s_3(t)$ must be not positive, otherwise the largest positive root of $s_3(t)$ would be a root of $s(t)$ (which doesn't have positive roots) since this root
couldn't be canceled by a root of $s_5(t)$ or $s_3(t+1)$.

Now,
\[
q(t) = r(t)r(t+1)\cdots r(t + k - 1) = t(t+1)\cdots(t+k-1)\frac{s_3(t)s_5(t+k)}{s_3(t+k)s_5(t)}
\]
is a polynomial. If $s_3(t) \not\in K$ and its smallest root is $i$, where $i \leq 0$, then the denominator of $q(t)$ has a zero in $i-k$ which is less then $1 - k$ and cannot be canceled by a zero in the numerator since $s_3(t+k)$ and $s_5(t+k)$ are relatively prime.

Hence
\[
s_3(t) \in K, \qquad
s(t) = \frac{s_5(t+1)}{s_5(t)},\qquad \text{and}\qquad
q(t) = t(t+1)\cdots(t+k-1) \frac{s_5(t+k)}{s_5(t)}.
\]

We can uniquely write
\[
s_5(t) = \prod_{i \in I}\phi_{k, p_i}(t+i), \qquad \text{where}\quad
\phi_{k,p}(t) = \prod_{j=0}^{p} (t + jk),
\]
and all $p_i$ are maximal possible.
Then
\[
\frac{s_5(t+k)}{s_5(t)} = \prod_{i \in I} \frac{t + i + p_ik + k}{t + i}\qquad
\text{and}\qquad
t + i_1 \neq t + i_2 + p_{i_2}k + k
\]
for all $i_1, i_2 \in I$ because of the
maximality of $p_i$. Hence $I \subset \{1, \dots, k-1\}$ and each $i$ is used at most once.

As we have seen, all roots of $s_5(t)$ are of multiplicity $1$
and since
\[
(xr)^N = x^N\prod_{i=0}^{N-1}(t+i)\frac{s_5(t+N)}{s_5(t)}
\]
the elements $(xr)^N \in K[x]$ for sufficiently large $N$. Therefore the rank of $C(xr)$ is one.
In~fact, we see that if $s_5$ is any polynomial with only simple roots then the elements of $A_1$ which commute with $xt\frac{s_5(t+1)}{s_5(t)}$ form a centralizer of the rank one.

Observe that the rank is not stable under automorphisms: the rank of $C(\phi(xr))$, where $\phi(x) = x + t^M$, $\phi(t) = t$ is $M + 1$.

Let us return now to $C(a)$, where $a = x^{m-n}p(t)$, $\deg_t(p(t)) = n$. In this case
\begin{gather*}
b = x^{d-1}t\frac{s_5(t+d-1)}{s_5(t)},
\qquad
b^n = x^{m-n}t(t + d -1)\cdots(t + (n-1)(d-1))\frac{s_5(t+n(d-1))}{s_5(t)}
\end{gather*}
and all roots of $s_5(t)$ belong to $\{1 - d, 2(1 - d), \dots, (n-1)(1-d)\}$.

Additionally we can replace $t$ by $t - c$, $c \in K$.

\section{Cases 2 and 3}

We understood the structure of $C(a)$ when $a = x^{m-n}p(t)$. Are there substantially different examples of centralizers of rank one which are not isomorphic to a polynomial ring?

Consider the case of an order $2$ element commuting with an order $3$ element which was completely researched in
the work~\cite{BC1} of Burchnall and Chaundry for analytic coefficients.\footnote{To be on a more familiar ground in this section the field $K$ is the filed $\C$ of complex numbers.} They showed (and for this case it is a straightforward computation) that monic commuting operators of orders $2$ and $3$ can be reduced to
\[
A = \de^2 - 2\psi(x),\qquad
B = \de^3 - 3\psi(x)\de - \frac{3}{2}\psi'(x),
\]
where $\psi''' = 12\psi\psi'$, i.e., $\psi'' = 6\psi^2 + c_1$ and $(\psi')^2 = 4\psi^3 + c_1\psi + c_2$ (a Weierstrass function). The only rational (even algebraic) solution in this case is (up to a substitution) $\psi = x^{-2}$ when $c_1 = c_2 = 0$. (If $\psi$ is a rational function then the curve parameterized by $\psi$, $\psi'$ has genus zero, so $4\psi^3 + c_1\psi + c_2 = 4(\psi - \la)^2(\psi - \mu)$ and $\psi$ is not an algebraic function of $x$ if $\la \ne \mu$.) The corresponding operator
\[
A = x^{-2}(t-2)(t+1) = \bigg(x^{-1}\frac{t^2 - 1}{t}\bigg)^2
\]
is homogeneous.

In our case we have
\[
A = f(x)^2\de^2 + f_1(x)\de + f_2(x)
\]
since the leading form for $w(x) = 0$, $w(\de) = 1$ must be the square of a polynomial.

Here are computations for this case
\begin{gather*}
A = (f\de)^2 - ff'\de + f_1(x)\de + f_2(x)
\\ \phantom{A}
{}= \bigg[f\de + \frac{1}{2}\bigg(\frac{f_1}{f} - f'\bigg)\bigg]^2 - \frac{1}{2}\bigg(\frac{f_1}{f} - f'\bigg)'f - \frac{1}{4}\bigg(\frac{f_1}{f} - f'\bigg)^2 + f_2.
\end{gather*}
Denote $f\de + \frac{1}{2}\big(\frac{f_1}{f} - f'\big)$ by $D$. Then
\[
A = D^2 - 2\phi (x), \qquad
\text{where}\quad
\phi = \frac{1}{4}\bigg(\frac{f_1}{f} - f'\bigg)'f + \frac{1}{8}\bigg(\frac{f_1}{f} - f'\bigg)^2 - \frac{1}{2}f_2 \in \C(x).
\]

Analogously to Burchnall and Chaundry, if there is an operator of order $3$ commuting with~$A$ then it can be written as
\[
B = D^3 - 3\phi D - \frac{3}{2}\phi'f
\]
(this follows from~\cite{BC1} but will be clear from the condition $[A, B] = 0$ as well).
In order to find an equation for $\phi$ we should compute $[A, B]$. Observe that
\begin{gather*}
[D, g(x)] = g'f, \qquad
\big[D^2, g\big] = 2g'fD + (g'f)'f, \\
\big[D^3, g\big] = 3g'fD^2 + 3(g'f)'fD + ((g'f)'f)'f.
\end{gather*}
Hence
\begin{gather*}
[A,B] = -3\bigg[D^2, \phi D + \frac{1}{2}\phi'f\bigg] + 2[D^3 - 3\phi D, \phi]
\\ \hphantom{[A,B]}
{}= -3\bigg[(2\phi'fD + (\phi'f)'f)D + (\phi'f)'fD + \frac{1}{2}((\phi'f)'f)'f\bigg]
\\ \hphantom{[A,B]=}
{}+ 2[3\phi'fD^2 + 3(\phi'f)'fD + ((\phi'f)'f)'f] - 6\phi\phi'f
\\ \hphantom{[A,B]}
{}= (-6\phi'f\! + \!6\phi'f)D^2 \!+\! (-6(\phi'f)'f \!+ \!6(\phi'f)'f)D
\!-\! \frac{3}{2}((\phi'f)'f)'f\!+ \!2((\phi'f)'f)'f\! -\! 6\phi\phi'f
\\ \hphantom{[A,B]}
{}= \frac{1}{2}((\phi'f)'f)'f - 6\phi\phi'f.
\end{gather*}

Therefore
\begin{gather*}
((\phi'f)'f)' = 12\phi\phi', \qquad
(\phi'f)'f = 6\phi^2 + c_1,
\\
(\phi'f)'\phi'f = 6\phi^2\phi' + c_1\phi',\qquad
(\phi'f)^2 = 4\phi^3 + 2c_1\phi + c_2
\end{gather*}
 and we have a parameterization of an elliptic curve. Since $f, \phi \in \C(x)$ this curve must have genus $0$, i.e.,
 \[
 4\phi^3 + 2c_1\phi + c_2 = 4(\phi - \la)^2(\phi - \mu)\qquad \text{and}\qquad (\phi'f)^2 = 4(\phi - \la)^2(\phi - \mu).
 \]

Take $z = \frac{\phi'f}{2(\phi - \la)}$. Then $\phi - \mu = z^2$ and $\phi'f = 2z\big(z^2 - \delta^2\big)$, where $\delta^2 = \la - \mu$. Hence $\phi' = 2zz'$, $2zz'f = 2z\big(z^2 - \delta^2\big)$ and $z'f = z^2 - \delta^2$.

Assume that $\delta \ne 0$. Since we can re-scale $f$ and $z$ as $f \rightarrow 2\delta f$, $z \rightarrow \delta z$, let us further assume that $\delta^2 = 1$.
Then
\[
z'f = z^2 - 1\qquad\text{and}\qquad
\int \frac{{\rm d} z}{z^2 - 1} = \int \frac{{\rm d} x}{f}.
\]
Recall that $f \in \C[x]$. Since
\[
2\int\frac{{\rm d} z}{z^2 - 1} = \ln\frac{z - 1}{z + 1}
\]
all zeros of $f$ have multiplicity $1$ and
\[\int \frac{{\rm d} x}{f} = \ln\bigg(\prod_i(x - \nu_i)^{c_i}\bigg),\] where $\{\nu_i\}$ are the roots of $f$ and $c_i = (f'(\nu_i))^{-1}$. Therefore
\[
\frac{z - 1}{z + 1} = c \prod_i(x - \nu_i)^{2c_i}, \qquad\text{where}\quad
c \ne 0\quad\text{and}\quad
z = \frac{1 + c \prod_i(x - \nu_i)^{2c_i}}{1 - c \prod_i(x - \nu_i)^{2c_i}}.
\]

Now it is time to recall that
\[
\phi = \frac{1}{4}\bigg(\frac{f_1}{f} - f'\bigg)'f + \frac{1}{8}\bigg(\frac{f_1}{f} - f'\bigg)^2 - \frac{1}{2}f_2,
\]
where $f, f_1, f_2 \in \C[x]$ and thus $f^2\phi = f^2\big(z^2 + \mu\big) \in \C[x]$.
Because of that{\samepage
\[
zf = c_1\frac{1 + c \prod_i(x - \nu_i)^{2c_i}}{1 - c \prod_i(x - \nu_i)^{2c_i}}\prod_i (x - \nu_i) \in \C[x],
\]
which is possible only if the rational function $1 - c \prod_i(x - \nu_i)^{2c_i}$ doesn't have zeros.}

We can write $\prod_i(x - \nu_i)^{2c_i}$ as $\frac{\prod_j(x - \nu_j)^{2c_j}}{\prod_k(x - \nu_k)^{2c_k}}$, where $2c_j, 2c_k \in \Z^+$.
Then
\[
\prod_k(x - \nu_k)^{2c_k} - c\prod_j(x - \nu_j)^{2c_j} \in \C,
\]
which is possible only if $c = 1$.

Since
\[
z = \frac{\prod_k(x - \nu_k)^{2c_k} + \prod_j(x - \nu_j)^{2c_j}}{\prod_k(x - \nu_k)^{2c_k} - \prod_j(x - \nu_j)^{2c_j}}
\]
we see that $z \in \C[x]$. So to produce a $2,3$ commuting pair we should find a polynomial solution to $f = \frac{z^2 - 1}{z'}$.
If $f$, $z$ are given then
\[
A = (f\de + \psi)^2 - 2\big(z^2 - \mu\big), \qquad
B = (f\de + \psi)^3 - 3\big(z^2 - \mu\big)(f\de + \psi) - 3zz'f
\]
is a commuting pair for any $\psi \in \C[x]$ (indeed, $f\psi \in \C[x]$ and $\psi^2 + f\psi' \in \C[x]$, hence $\psi \in \C[x]$).
Constant $\mu = -\frac{2}{3}$ since we assumed that $\la - \mu = 1$ and $2\la + \mu = 0$ because the equation is $(\phi'f)^2 = 4\phi^3 + 2c_1\phi + c_2$.

Here is a series of examples:
\[
z = 1 + x^n,\qquad
f = \frac{x}{n}\big(2 + x^n\big),\qquad
\phi = \big(1 + x^n\big)^2 - \frac{2}{3} = x^n(2 + x^n) +\frac{1}{3},
\]
which correspond to
\[
A = \bigg[\frac{x}{n}(2 + x^n)\de + \psi\bigg]^2 - 2\bigg(x^n(2 + x^n) + \frac{1}{3}\bigg).
\]
Even the simplest one,
\[
A = [x(2 + x)\de]^2 - 2\bigg[x(2 + x) + \frac{1}{3}\bigg]
\]
cannot be made homogeneous.

It~seems that a complete classification of $(2,3)$ pairs is a daunting task. Our condition on $z$ is that $z$ assumes values $\pm 1$ when $z' = 0$. Let us call such a polynomial admissible. We can look only at reduced monic polynomials $z(x) = x^n + a_2x^{n-2} + \cdots$ because a substitution $x \rightarrow ax + b$ preserves admissibility. Also $\la^nz(\la^{-1}x)$ preserves admissibility if $\deg(z) = n$ and $\la^n = 1$.

Examples above are just one value case. Say, an admissible cubic polynomial is $x^3 - 3\cdot2^{-\frac{2}{3}}x$. If~$z = (x-\nu)^i(x+\nu)^j +1$ then it is admissible when
\[
\nu^{i+j} = (-1)^{i-1}2^{1 -i-j}\frac{(i+j)^{i+j}}{i^ij^j}.
\]
If~a~com\-position $h(g(x))$ is admissible then $g' = 0$ and $h' = 0$ should imply that $h(g(x)) = \pm1$. Hence~$h(x)$ should be an admissible function. As far as $g$ is concerned $g' = 0$ should imply that the value of $g$ belongs to the preimage of $\pm 1$ for $h$ which is less restrictive if this preimage is large.
Because of that it is hard to imagine a reasonable classification of all admissible polynomials.

On the other hand
\[
z^2 \equiv 1 \pmod {z'}\qquad\text{for}\quad
z = x^n + a_2x^{n-2} + \cdots + a_n
\]
leads to $n-1$ equations on~$n-1$ variables with apparently finite number of solutions for each~$n$. Say, for $n = 4$ all admissible polynomials are
\[
x^4 \pm 1; \qquad
x^4 + ax^2 + \frac{1}{8}a^2, \quad
a^4 = 64; \qquad
x^4 - 3a^2x^2 + 2\sqrt{2}a^3x + \frac{21}{8}a^4,\quad
337a^8 = 64.
\]

\begin{Remark}\label{corres}
The number of admissible polynomials of a given degree is finite. Indeed consider first $n-2$ homogeneous equations on the coefficients $a_2, \dots, a_n$. They are satisfied if $z^2 \equiv c \pmod{z'}$, where $c \in \C$. If one of the components of the variety defined by these equations is more than one-dimensional then (by affine dimension theorem) its intersection with the hypersurface given by the last homogeneous equation will be at least one-dimensional while condition \mbox{$z^2 \equiv 0 \pmod{z'}$} is satisfied only by $z = x^n$ (recall that we are considering only reduced monic polynomials).
\end{Remark}

If $\delta = 0$ then $(\phi'f)^2 = 4(\phi - \la)^3$ and $(\phi - \la)^{-1/2} = \pm\int \frac{{\rm d} x}{f}$ is a rational function.

\begin{Lemma}\label{div-3}%
If $f\!\in\!\C[x]$ and $\int\! \frac{{\rm d} x}{f}$ is a rational function then $f$ is a monomial, i.e., $f\! = a(x \!- b)^d$.%
\footnote{I was unable to find a published proof for this observation. This proof is a result of discussions with J.~Bernstein and A.~Volberg.}
\end{Lemma}

\begin{proof}
If $g' = \frac{1}{f}$ for $g \in \C(x)$ then $g = \frac{h}{f}$, $h \in \C[x]$ since the poles of $g$ are the zeros of $f$ and if the multiplicity of a zero of $f$ is $d$ then the corresponding pole of $g$ has the multiplicity $d - 1$. An equality $g' = \frac{1}{f}$ can be rewritten as $h'f - hf' = f$. If $\deg(h) > 1$ then $\deg(h'f) > \deg(f)$. Hence the leading coefficients of polynomials $h'f$ and $hf'$ are the same. This is possible only when $\deg(h) = \deg(f)$. Therefore there exists a $c \in \C$ for which $\deg(h - cf) < \deg(f)$. Since $(h - cf)'f - (h - cf)f' = f$ we can conclude that $\deg(h_1) = 1$ for $h_1 = h - cf$. Changing the variable we may assume that $h_1 = c_1x$ and then $c_1(f - xf') = f$ which is possible only if~$f = ax^d$.
\end{proof}

Hence when $\delta = 0$ we may assume that $f = x^d$. If $d = 0$ then this is the first case and $A$ is a~homogeneous operator up to an automorphism. If $d > 0$ then this is the second case and $A$ is a~homogeneous operator up to an automorphism of $\C\big[x^{-1},x, \de\big]$.

These computations show that a description of the structure of centralizers of rank one in $A_1$ is sufficiently challenging.
Can the ring of regular functions of a genus zero curve with one place at infinity be realised as a centralizer of an element of $A_1$? Here is a more approachable relevant question: is there an element of $A \in D_1 \setminus A_1$ for which $p(A) \in A_1$ for a given a polynomial $p(x) \in \C[x]$?

\section{Historical remarks}

Apparently the first work which was devoted to the research of commuting differential operators is the work of
Georg Wallenberg ``\"Uber die Vertauschbarkeit homogener linearer Differentialausdr\"ucke'' (see~\cite{W}),
in which he studied the classification problem of pairs of commuting ordinary differential operators. He didn't work with the Weyl algebra though.
Differential ope\-rators he was working with have coefficients which are ``abstract'' differentiable functions.

He mentioned that this problem did not seem to be studied before, even in the fundamental work of Gaston Floquet Sur la th\'eorie des \'equations diff\'erentielles lin\'eaires (see~\cite{F}).
He credited Floquet with the case of two operators of order one.

Wallenberg started from this point. Then he gave a complete description of commuting operators $P$ and $Q$ when $\ord P = \ord Q = 2$ and when $\ord P = 1$, $\ord Q = n$. So far everything is easy. Then he studied the case of $\ord P = 2$ and $\ord Q = 3$, and noticed that a Weierstrass elliptic function appears in the coefficients of these operators.
He dealt with a few more examples such as order 2 and 5, but did not obtain any general theorems.

Issai Schur read the paper of Wallenberg and published in 1904 paper ``\"Uber vertauschbare line\-are Differentialausdr\"ucke'', which was already mentioned. He proved that centralizers of~differential operators with differentiable coefficients are commutative by~introducing pseudo-dif\-fe\-ren\-tial operators approximately fifty years before the notion appeared under this name.

The first general results toward the classification of commuting pairs of differential operators which were reported to the London Mathematical Society on June 8, 1922, were established about 20 years later by Joseph Burchnall and Theodore Chaundy (see~\cite{BC1}).

Burchnall and Chaundy published two more papers~\cite{BC2} and~\cite{BC3} devoted to this topic. Curiously enough they didn't know about the Wallenberg's paper and rediscovered his classification of~$2$,~$3$ commuting pairs and the fact that the ring of operators commuting with this pair of operators is isomorphic to a ring of regular functions of an elliptic curve.

Arguably the most important fact obtained by them is that an algebraic curve can be related to a pair of commuting operators.

After these works the question about commuting pairs of operators didn't attract much attention until the work of Jacques Dixmier~\cite{D}
which appeared in 1968. He found elements in $A_1$ with centralizers which are also isomorphic to the ring of regular functions of an elliptic curve. Unlike examples by Wallenberg and Burchnall and Chaundy, where it was a question of~rather straightforward computations, Dixmier's example required ingenuity.

After another gap of approximately ten years the question on pairs of commuting operators was raised in the context of solutions of some important partial differential equations.
It~seems that Igor Krichever was the first to write on this topic in~\cite{K1} which appeared in 1976. He also rediscovered that an algebraic curve can be associated to a pair of commuting operators (and attributes to A.~Shabat first observation of this kind) and mentions that operators commuting with an operator commute with each other.

Then Vladimir Drinfeld in~\cite{Dr} gave algebro-geometric interpretation of Krichever's results. This approach was further elucidated in a report of David Mumford~\cite{Mum} (D.~Kajdan mentioned at the end
of the report is David Kazhdan).

Later Krichever wrote an important survey~\cite{K2} devoted to application of algebraic geometry to solutions of nonlinear PDE.

These works were primarily concerned with centralizers of rank one.

In~\cite{K3} Krichever considered centralizers of an arbitrary rank and proved that for any centralizer $A$ of a (non-constant) differential operator there exists a marked algebraic curve $(\ga, P)$
such that $A$ is isomorphic to a ring of meromorphic functions on $\ga$ with poles in $P$. He remarked that $\ga$ is non-singular for a ``general position'' centralizer.

Motohico Mulase in~\cite{Mul} generalized the results of Krichever in~\cite{K1}, Drinfeld, and Mumford to the case of arbitrary rank. His main theorem is similar to the theorem of Krichever cited above.
Apparently he didn't know about~\cite{K3}.

It~remains to mention the works devoted to the centralizers of elements of the first Weyl algebra or its skew field of fractions. They primarily provide constructions of examples of~centralizers which correspond to the curves of high genus.

Here is a partial list:~\cite{DM, De, Gri, Gru, Mi6, Mi5, Mi4, Mi3, Mi2, Mi1, MZ, Mo1, Mo3, Mo2, Mo4, O1, O2, O4, O3, P2, ZM, ZMS}.

Interested reader will also benefit by looking at lectures by Emma Previato~\cite{P1} and at the paper~\cite{W3} by George Wilson as well as papers~\cite{SW, V, W2}, and~\cite{BZ}.

Lastly, it seems that the definition of rank of a centralizer as the greatest common divisor of the orders of its elements first appeared in a work of Wilson~\cite{W1}. A similar definition of rank formulated slightly differently can be found in the work~\cite{Dr} of Drinfeld.

\subsection*{Acknowledgements}

The author is grateful to the Max-Planck-Institut f\"{u}r Mathemtik in Bonn where he worked on~this project in July--August of 2019.
He was also supported by a FAPESP grant awarded by~the State of Sao Paulo, Brazil.
The author also greatly benefited from remarks by the referees.

\pdfbookmark[1]{References}{ref}
\LastPageEnding


\begin{thebibliography}{99}
\footnotesize\itemsep=0pt

\bibitem{A}
Amitsur S.A., Commutative linear differential operators, \href{https://doi.org/10.2140/pjm.1958.8.1}{\textit{Pacific~J.
 Math.}} \textbf{8} (1958), 1--10.

\bibitem{BZ}
Burban I., Zheglov A., Fourier--{M}ukai transform on {W}eierstrass cubics and
 commuting differential operators, \href{https://doi.org/10.1142/S0129167X18500647}{\textit{Internat.~J. Math.}} \textbf{29}
 (2018), 1850064, 46~pages, \href{https://arxiv.org/abs/1602.08694}{arXiv:1602.08694}.

\bibitem{BC1}
Burchnall J.L., Chaundy T.W., Commutative ordinary differential operators,
 \href{https://doi.org/10.1112/plms/s2-21.1.420}{\textit{Proc. London Math. Soc.}} \textbf{21} (1923), 420--440.

\bibitem{BC2}
Burchnall J.L., Chaundy T.W., Commutative ordinary differential operators, \href{https://doi.org/10.1098/rspa.1928.0069}{\textit{Proc. Roy. Soc. London~A}}
 \textbf{118} (1928), 557--583.

\bibitem{BC3}
Burchnall J.L., Chaundy T.W., Commutative ordinary differential operators~{II}.
 The identity {$P^n = Q^m$}, \href{https://doi.org/10.1098/rspa.1931.0208}{\textit{Proc. Roy. Soc. London~A}} \textbf{134}
 (1931), 471--485.

\bibitem{DM}
Davletshina V.N., Mironov A.E., On commuting ordinary differential operators
 with polynomial coefficients corresponding to spectral curves of genus two,
 \href{https://doi.org/10.4134/BKMS.b160685}{\textit{Bull. Korean Math. Soc.}} \textbf{54} (2017), 1669--1675,
 \href{https://arxiv.org/abs/1606.01346}{arXiv:1606.01346}.

\bibitem{De}
Dehornoy P., Op\'erateurs diff\'erentiels et courbes elliptiques,
 \textit{Compositio Math.} \textbf{43} (1981), 71--99.

\bibitem{D}
Dixmier J., Sur les alg\`ebres de {W}eyl, \href{https://doi.org/10.24033/bsmf.1667}{\textit{Bull. Soc. Math. France}}
 \textbf{96} (1968), 209--242.

\bibitem{Dr}
Drinfel'd V.G., Commutative subrings of certain noncommutative rings,
 \href{http://dx.doi.org/10.1007/BF01135527}{\textit{Funct. Anal. Appl.}} \textbf{11} (1977), 9--12.

\bibitem{F}
Floquet G., Sur la th\'eorie des \'equations diff\'erentielles lin\'eaires,
 \href{https://doi.org/10.24033/asens.182}{\textit{Ann. Sci. \'Ecole Norm. Sup.~(2)}} \textbf{8} (1879), 3--132.

\bibitem{GK}
Gelfand I.M., Kirillov A.A., Sur les corps li\'es aux alg\`ebres enveloppantes
 des alg\`ebres de {L}ie, \href{https://doi.org/10.1007/BF02684800}{\textit{Inst. Hautes \'Etudes Sci. Publ. Math.}}
 (1966), 5--19.

\bibitem{Gri}
Grinevich P.G., Rational solutions for the equation of commutation of
 differential operators, \href{https://doi.org/10.1007/BF01081803}{\textit{Funct. Anal. Appl.}} \textbf{16} (1982),
 15--19.

\bibitem{Gru}
Gr\"unbaum F.A., Commuting pairs of linear ordinary differential operators of
 orders four and six, \href{https://doi.org/10.1016/0167-2789(88)90007-3}{\textit{Phys.~D}} \textbf{31} (1988), 424--433.

\bibitem{K1}
Krichever I.M., An algebraic-geometric construction of the Zakharov--Shabat
 equations and their periodic solution, \textit{Sov. Math. Dokl.} \textbf{17}
 (1976), 394--397.

\bibitem{K2}
Krichever I.M., Methods of algebraic geometry in the theory of non-linear
 equations, \href{https://doi.org/10.1070/RM1977v032n06ABEH003862}{\textit{Russian Math. Surveys}} \textbf{32} (1977), no.~6,
 185--213.

\bibitem{K3}
Krichever I.M., Commutative rings of ordinary linear differential operators,
 \href{https://doi.org/10.1007/BF01681429}{\textit{Funct. Anal. Appl.}} \textbf{12} (1978), 175--185.

\bibitem{ML}
Makar-Limanov L., Centralizers in the quantum plane algebra, in Studies in
 {L}ie Theory, \textit{Progr. Math.}, Vol. 243, \href{https://doi.org/10.1007/0-8176-4478-4_16}{Birkh\"auser Boston}, Boston,
 MA, 2006, 411--416.

\bibitem{Mi6}
Mironov A.E., A~ring of commuting differential operators of rank~2
 corresponding to a~curve of genus~2, \href{https://doi.org/10.1070/SM2004v195n05ABEH000823}{\textit{Sb. Math.}} \textbf{195} (2004),
 103--114.

\bibitem{Mi5}
Mironov A.E., Commuting rank~2 differential operators corresponding to a curve
 of genus~2, \href{https://doi.org/10.1007/s10688-005-0045-1}{\textit{Funct. Anal. Appl.}} \textbf{39} (2005), 240--243.

\bibitem{Mi4}
Mironov A.E., On commuting differential operators of rank~2, \textit{Sib.
 Electr. Math. Rep.} \textbf{6} (2009), 533--536.

\bibitem{Mi3}
Mironov A.E., Commuting higher rank ordinary differential operators, in
 European {C}ongress of {M}athematics, \href{https://doi.org/10.4171/120-1/27}{Eur. Math. Soc.}, Z\"urich, 2013,
 459--473, \href{https://arxiv.org/abs/1204.2092}{arXiv:1204.2092}.

\bibitem{Mi2}
Mironov A.E., Self-adjoint commuting ordinary differential operators,
 \href{https://doi.org/10.1007/s00222-013-0486-8}{\textit{Invent. Math.}} \textbf{197} (2014), 417--431.

\bibitem{Mi1}
Mironov A.E., Self-adjoint commuting differential operators of rank two,
 \href{http://dx.doi.org/10.1070/RM9730}{\textit{Russian Math. Surveys}} \textbf{71} (2016), 751--779.

\bibitem{MZ}
Mironov A.E., Zheglov A.B., Commuting ordinary differential operators with
 polynomial coefficients and automorphisms of the first {W}eyl algebra,
 \href{https://doi.org/10.1093/imrn/rnv218}{\textit{Int. Math. Res. Not.}} \textbf{2016} (2016), 2974--2993,
 \href{https://arxiv.org/abs/1503.00485}{arXiv:1503.00485}.

\bibitem{Mo1}
Mokhov O.I., Commuting ordinary differential operators of rank~3 corresponding
 to an elliptic curve, \href{https://doi.org/10.1070/RM1982v037n04ABEH003953}{\textit{Russian Math. Surveys}} \textbf{37} (1982),
 no.~4, 129--130.

\bibitem{Mo3}
Mokhov O.I., Commuting differential operators of rank~3 and nonlinear
 equations, \href{https://doi.org/10.1070/IM1990v035n03ABEH000720}{\textit{Math. USSR-Izv.}} \textbf{53} (1989), 629--655.

\bibitem{Mo2}
Mokhov O.I., On commutative subalgebras of the {W}eyl algebra related to
 commuting operators of arbitrary rank and genus, \href{https://doi.org/10.1134/S0001434613070298}{\textit{Math. Notes}}
 \textbf{94} (2013), 298--300, \href{https://arxiv.org/abs/1201.5979}{arXiv:1201.5979}.

\bibitem{Mo4}
Mokhov O.I., Commuting ordinary differential operators of arbitrary genus and
 arbitrary rank with polynomial coefficients, in Topology, Geometry,
 Integrable Systems, and Mathematical Physics, \textit{Amer. Math. Soc.
 Transl. Ser.~2}, Vol.~234, \href{https://doi.org/10.1090/trans2/234/16}{Amer. Math. Soc.}, Providence, RI, 2014, 323--336,
 \href{https://arxiv.org/abs/1303.4263}{arXiv:1303.4263}.

\bibitem{Mul}
Mulase M., Category of vector bundles on algebraic curves and
 infinite-dimensional {G}rassmannians, \href{https://doi.org/10.1142/S0129167X90000174}{\textit{Internat.~J. Math.}} \textbf{1}
 (1990), 293--342.

\bibitem{Mum}
Mumford D., An algebro-geometric construction of commuting operators and of
 solutions to the {T}oda lattice equation, {K}orteweg--de {V}ries equation and
 related nonlinear equation, in Proceedings of the {I}nternational {S}ymposium
 on {A}lgebraic {G}eometry ({K}yoto {U}niv., {K}yoto, 1977), Kinokuniya Book
 Store, Tokyo, 1978, 115--153.

\bibitem{O1}
Oganesyan V.S., Commuting differential operators of rank~2 and arbitrary
 genus~$g$ with polynomial coefficients, \href{https://doi.org/10.1070/RM2015v070n01ABEH004939}{\textit{Russian Math. Surveys}}
 \textbf{70} (2015), 165--167.

\bibitem{O2}
Oganesyan V.S., Commuting differential operators of rank~2 with polynomial
 coefficients, \href{https://doi.org/10.1007/s10688-016-0128-1}{\textit{Funct. Anal. Appl.}} \textbf{50} (2016), 54--61,
 \href{https://arxiv.org/abs/1409.4058}{arXiv:1409.4058}.

\bibitem{O4}
Oganesyan V.S., An alternative proof of {M}ironov's results on self-adjoint
 commuting operators of rank~2, \href{https://doi.org/10.1134/S0037446618010111}{\textit{Sib. Math.~J.}} \textbf{59} (2018),
 102--106.

\bibitem{O3}
Oganesyan V.S., Commuting differential operators of rank~2 with rational
 coefficients, \href{https://doi.org/10.1007/s10688-018-0229-0}{\textit{Funct. Anal. Appl.}} \textbf{52} (2018), 203--213,
 \href{https://arxiv.org/abs/1608.05146}{arXiv:1608.05146}.

\bibitem{P1}
Previato E., Seventy years of spectral curves: 1923--1993, in Integrable
 Systems and Quantum Groups ({M}ontecatini {T}erme, 1993), \textit{Lecture
 Notes in Math.}, Vol.~1620, \href{https://doi.org/10.1007/BFb0094795}{Springer}, Berlin, 1996, 419--481.

\bibitem{P2}
Previato E., Rueda S.L., Zurro M.A., Commuting ordinary differential operators
 and the {D}ixmier test, \href{https://doi.org/10.3842/SIGMA.2019.101}{\textit{SIGMA}} \textbf{15} (2019), 101, 23~pages,
 \href{https://arxiv.org/abs/1902.01361}{arXiv:1902.01361}.

\bibitem{S}
Schur I., \"Uber vertauschbare lineare Differentialausdr\"ucke,
 \textit{Sitzungsber. Berl. Math. Ges.} (1905), 2--8.

\bibitem{SW}
Segal G., Wilson G., Loop groups and equations of {K}d{V} type, \href{https://doi.org/10.1007/BF02698802}{\textit{Inst.
 Hautes \'Etudes Sci. Publ. Math.}} (1985), 5--65.

\bibitem{V}
Verdier J.-L., \'Equations diff\'erentielles alg\'ebriques, in S\'eminaire
 {B}ourbaki, 30e ann\'ee (1977/78), \textit{Lecture Notes in Math.}, Vol.~710,
 \href{https://doi.org/10.1007/BFb0069975}{Springer}, Berlin, 1979, Exp. No.~512, 101--122.

\bibitem{W}
Wallenberg G., \"Uber die Vertauschbarkeit homogener linearer
 Differentialausdr\"ucke, \textit{Arch. Math. Phys.} \textbf{3} (1903),
 252–268.

\bibitem{W1}
Wilson G., Algebraic curves and soliton equations, in Geometry Today ({R}ome,
 1984), \textit{Progr. Math.}, Vol.~60, Birkh\"auser Boston, Boston, MA, 1985,
 303--329.

\bibitem{W2}
Wilson G., Bispectral commutative ordinary differential operators,
 \href{https://doi.org/10.1515/crll.1993.442.177}{\textit{J.~Reine Angew. Math.}} \textbf{442} (1993), 177--204.

\bibitem{W3}
Wilson G., Collisions of {C}alogero--{M}oser particles and an adelic
 {G}rassmannian (with an appendix by I.G.~Macdonald), \href{https://doi.org/10.1007/s002220050237}{\textit{Invent. Math.}}
 \textbf{133} (1998), 1--41.

\bibitem{ZM}
Zheglov A.B., Mironov A.E., On commuting differential operators with polynomial
 coefficients corresponding to spectral curves of genus one, \href{https://doi.org/10.1134/s1064562415030126}{\textit{Dokl.
 Math.}} \textbf{91} (2015), 281--282.

\bibitem{ZMS}
Zheglov A.B., Mironov A.E., Saparbaeva B.T., Commuting {K}richever--{N}ovikov
 differential operators with polynomial coefficients, \href{https://doi.org/10.1134/s0037446616050104}{\textit{Sib. Math.~J.}}
 \textbf{57} (2016), 819--823.

\end{thebibliography}
\end{document}